\documentclass[12pt]{article}
\usepackage[margin=1in]{geometry}
\usepackage{amsthm,amsmath,amsfonts,amssymb,hyperref,tikz,mathrsfs,scrextend,url,authblk}
\newtheorem{theorem}{Theorem}[section]
\newtheorem{lemma}[theorem]{Lemma}
\newtheorem{proposition}[theorem]{Proposition}
\newtheorem{definition}[theorem]{Definition}

\newtheorem{remark}[theorem]{Remark}

\numberwithin{equation}{section}
\DeclareMathOperator{\cl}{cl}

\DeclareMathOperator{\irr}{irr}

\numberwithin{equation}{section}
\title{Polymatroids, closure operators and lattices}
\author{William Gustafson}
\affil{University of Kentucky\\ Department of Mathematics\\ Lexington KY 40506-0027 USA\\
\texttt{william.gustafson@uky.edu}}

\begin{document}
\maketitle

\begin{abstract}
	We study the closure operators of polymatroids from a lattice theoretic
	point of view. We show that polymatroid closure operators relate to
	lattices enriched with a generating set in the same way that matroids
	relate to geometric lattices. Through this relation we define
	a notion of minors for lattices enriched with a generating set.
	For the lattice of flats of a graphic matroid,
	the minors of the lattice are shown to correspond to simple minors of the graph
	when the vertices are labeled and the edges unlabeled.
	This correspondence is generalized to all polymatroids.
\end{abstract}

\section{Introduction}
	Polymatroids were introduced by Edmonds in~\cite{edmondsOrig} as a generalization
of matroids in connection with optimization theory. Edmonds introduced
polymatroids as certain polytopes lying in the nonnegative orthant of the real
vector space spanned by a ground set.
Vectors in the nonnegative orthant can be viewed as nonnegative real weightings
of the ground set, subsets of the ground set corresponding to 0,1 vectors.
Conceptually the points in a polymatroid encapsulate independent weightings.
Given a matroid~$M$ there is an associated polymatroid which is the convex hull of
the 0,1 vectors which correspond to independent sets of~$M$.
In this sense polymatroids are a generalization of matroids. Edmonds also
gave a description of polymatroids in terms of a rank function,
see~\cite[Theorem 14]{edmondsOrig}. We will prefer this rank function definition
given below. For a set~$E$ we let~$B_E$ denote the poset of all subsets of~$E$
ordered by inclusion.

\begin{definition}
	A \emph{polymatroid} on the ground set~$E$
	is a function~$r:B_E\rightarrow\mathbb{R}_{\ge0}$
	satisfying the following conditions for all~$X,Y\subseteq E$.
	\begin{align}
		&r(\emptyset)=0,\\
		&\text{If }X\subseteq Y\text{ then }r(X)\le r(Y),\\
		&\label{submodularity}r(X\cap Y)+r(X\cup Y)\le r(X)+r(Y).
	\end{align}
\end{definition}
		
Condition~\ref{submodularity} is referred to as \emph{submodularity} of the
function~$r$.
When the polymatroid~$r$ is integer valued, and for all~$X\subseteq E$ satisfies~$r(X)
\le\lvert{X}\rvert$ then it is (the rank function of) a matroid.

Many notions from matroid theory carry over directly, or nearly so,
when stated in terms of the rank function. Given a polymatroid~$r$ on~$E$
we say that an element~$e\in E$ is a \emph{loop} with respect to~$r$
when~$r(\{e\})=0$. Two elements~$e,f\in E$ are said to be \emph{parallel}
with respect to~$r$ when~$r(\{e,f\})=r(\{e\})=r(\{f\})$. The \emph{parallel
class of~$e$} with respect to~$r$ is the collection of elements in~$E$
which are parallel to~$e$.
A polymatroid is \emph{simple} if it has no loops and all parallel classes are trivial.

Our main interest in the present work is the closure operator of polymatroids.
The \emph{closure operator} of a polymatroid~$r$ on~$E$ is the map~$\overline{\ \cdot\ }:
B_E\rightarrow B_E$ defined for~$X\subseteq E$ by
\[
\overline{X}=\{e\in E:r(X)=r(X\cup\{e\})\}.
\]
The submodularity of~$r$ implies that~$r(\overline{X})=r(X)$
for any~$X\subseteq E$.
Sets of the form~$\overline{X}$ are referred to as~\emph{$r$-closed sets}
or \emph{flats} of~$r$. Edmonds showed that the set of flats of a polymatroid
is closed under intersection (\cite[Theorem 25]{edmondsOrig}). Since the set of
flats of a polymatroid is finite and has a maximal element, namely~$E$, this implies
that the set of flats ordered under inclusion forms a lattice. The meet in this
lattice is intersection and the join is given by~$X\vee Y=\overline{X\cup Y}$.

Given a matroid~$r$ the closure operator uniquely determines~$r$. For polymatroids
this is not the case as the closure operator has no information about how
much the rank of sets may differ. For example, define a polymatroid
on the ground set~$E=\{1,2\}$ by assigning rank value~1 to~$\{1\}$ and to~$\{2\}$
and assigning any rank value in the interval~$(1,2]$ to the set~$\{1,2\}$.
The resulting closure operator is the identity map on~$B_E$ regardless of the choice
of the rank of~$\{1,2\}$.

\section{Polymatroids and generator enriched lattices}
\label{polymatroids section}
	We now introduce an object which will be seen to correspond to polymatroid
closure operators in the same way geometric lattices correspond to matroids.

\begin{definition}
	A generator enriched lattice is a pair~$(L,G)$ in which~$L$ is a finite lattice
	and~$G\subseteq L\setminus\{\widehat{0}\}$ generates the lattice~$L$ via the
	join operation.
\end{definition}

Note that if~$(L,G)$ is a generator enriched lattice, the set~$G$ necessarily contains
the set of join irreducibles of~$L$ which we will denote as~$\irr(L)$.
A generator enriched lattice of the form~$(L,\irr(L))$ will be said to be \emph{minimally generated}.

A lattice is typically depicted via its Hasse diagram.
The Hasse diagram is not enough information to specify a generator enriched lattice
since it does not describe the generating set.
Instead a generator enriched lattice
may be depicted via a diagram analogous to Cayley graphs for groups
with a generating set. Given a generator enriched lattice~$(L,G)$ the
associated diagram has vertex set~$L$, and directed
edges~$(\ell,\ell\vee g)$ for~$\ell\in L$ and~$g\in G$
such that~$\ell\ne\ell\vee g$. Just as with Hasse diagrams all diagrams
of generator enriched lattices will be depicted so that the edges are directed upwards.
The diagram of a generator enriched lattice determines the underlying lattice:
the order relation~$\ell_1\le\ell_2$ holds when there is a directed path from~$\ell_1$ to~$\ell_2$
in the diagram. 
The minimal element~$\widehat{0}$ is the unique source vertex.
The generating set consists of the elements adjacent to~$\widehat{0}$.
See Figure~\ref{fig cayleys} for examples of diagrams of generator enriched lattices.

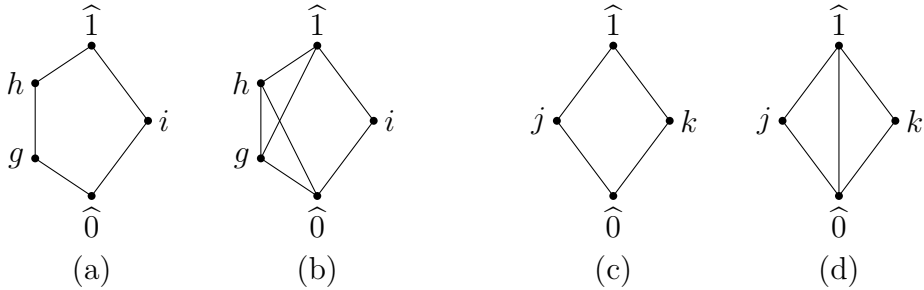
\begin{figure}[ht]
	\centering
	\begin{tikzpicture}
	\node at (0,0)
	{
	\begin{tikzpicture}
	\begin{scope}
		\fill (0,0) circle (1.5pt) coordinate (123);
		\fill (-.75,-1.5) circle (1.5pt) coordinate (1);
		\fill (-.75,-.5) circle (1.5pt) coordinate (2);
		\fill (.75,-1) circle (1.5pt) coordinate (3);
		\fill (0,-2) circle (1.5pt) coordinate (0);
		\node[below]at(0){$\widehat{0}$};
		\node[left]at(1){$g$};
		\node[left]at(2){$h$};
		\node[right]at(3){$i$};
		\node[above]at(123){$\widehat{1}$};
		\node at(0,-3) {(a)};

		\draw (0)--(1)--(2)--(123)--(3)--(0);
	\end{scope}
	\end{tikzpicture}
	};
	\node at (3,0)
	{
	\begin{tikzpicture}
	\begin{scope}
		\fill (0,0) circle (1.5pt) coordinate (123);
		\fill (-.75,-1.5) circle (1.5pt) coordinate (1);
		\fill (-.75,-.5) circle (1.5pt) coordinate (2);
		\fill (.75,-1) circle (1.5pt) coordinate (3);
		\fill (0,-2) circle (1.5pt) coordinate (0);
		\node[below]at(0){$\widehat{0}$};
		\node[left]at(1){$g$};
		\node[left]at(2){$h$};
		\node[right]at(3){$i$};
		\node[above]at(123){$\widehat{1}$};
		\node at(0,-3){(b)};

		\draw (0)--(1)--(2)--(123)--(3)--(0);
		\draw (0)--(2);
		\draw (1)--(123);
	\end{scope}
	\end{tikzpicture}
	};
	\node at (7,0)
	{
	\begin{tikzpicture}
	\begin{scope}
		\fill (0,0) circle (1.5pt) coordinate (123);
		\fill (-.75,-1) circle (1.5pt) coordinate (1);
		\fill (.75,-1) circle (1.5pt) coordinate (2);
		\fill (0,-2) circle (1.5pt) coordinate (0);
		\node[below]at(0){$\widehat{0}$};
		\node[left]at(1){$j$};
		\node[right]at(2){$k$};
		\node[above]at(123){$\widehat{1}$};
		\node at (0,-3) {(c)};

		\draw (0)--(1)--(123)--(2)--(0);
	\end{scope}
	\end{tikzpicture}
	};
	\node at (10,0)
	{
	\begin{tikzpicture}
	\begin{scope}
		\fill (0,0) circle (1.5pt) coordinate (123);
		\fill (-.75,-1) circle (1.5pt) coordinate (1);
		\fill (.75,-1) circle (1.5pt) coordinate (2);
		\fill (0,-2) circle (1.5pt) coordinate (0);
		\node[below]at(0){$\widehat{0}$};
		\node[left]at(1){$j$};
		\node[right]at(2){$k$};
		\node[above]at(123){$\widehat{1}$};
		\node at(0,-3){(d)};

		\draw (0)--(1)--(123)--(2)--(0)--(123);
	\end{scope}
	\end{tikzpicture}
	};
\end{tikzpicture}

	\caption{In (a) is the Hasse diagram of a lattice~$L$
	with~$\irr(L)=\{g,h,i\}$, and in (b) is
	the diagram of the associated minimally generated lattice~$(L,\irr(L))$.
	In (c) is the Hasse diagram of the Boolean algebra~$B_2$, which is also
	the diagram of the minimally generated lattice~$(B_2,\{j,k\})$,
	and in (d) is the diagram of the generator enriched lattice~$(B_2,\{j,k,\widehat{1}\})$.
	}
	\label{fig cayleys}
\end{figure}

For every polymatroid we have an associated generator enriched lattice.

\begin{definition}
	Given a polymatroid~$r:E\rightarrow\mathbb{R}_{\ge0}$ the \emph{generator enriched lattice of flats}
	is the generator enriched lattice~$(L,G)$ where
	\begin{align*}
		L&=\{\overline{X}:X\subseteq E\},\\
		G&=\{\overline{\{e\}}:e\in E,\ r(\{e\})\ne 0\}.
	\end{align*}
\end{definition}

See Figure~\ref{fig genlatt of flats} for examples of polymatroids and the associated
generator enriched lattice of flats.

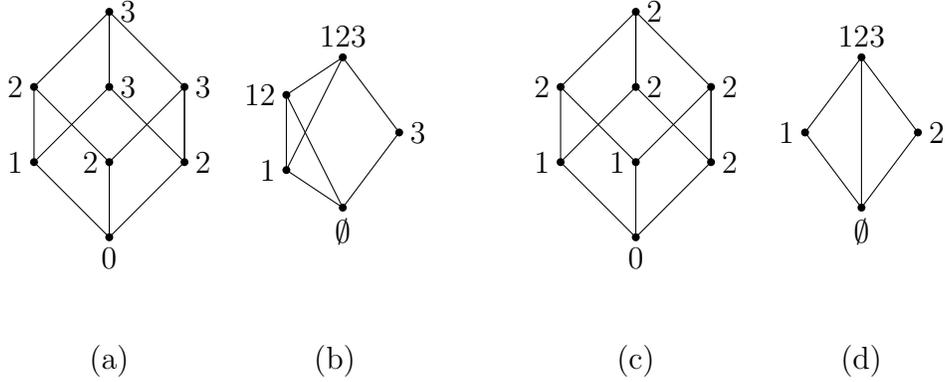
\begin{figure}[ht]
	\centering
	\begin{tikzpicture}
	\node at(0,0)
	{
	\begin{tikzpicture}
	\begin{scope}
		\fill(0,0)circle(1.5pt)coordinate(123);
		\fill(-1,-1)circle(1.5pt)coordinate(12);
		\fill(0,-1)circle(1.5pt)coordinate(13);
		\fill(1,-1)circle(1.5pt)coordinate(23);
		\fill(-1,-2)circle(1.5pt)coordinate(1);
		\fill(0,-2)circle(1.5pt)coordinate(2);
		\fill(1,-2)circle(1.5pt)coordinate(3);
		\fill(0,-3)circle(1.5pt)coordinate(e);
		\node[below]at(e){0};
		\node[left]at(1){1};
		\node[left]at(2){2};
		\node[right]at(3){2};
		\node[left]at(12){2};
		\node[right]at(13){3};
		\node[right]at(23){3};
		\node[right]at(123){3};
		\draw(e)--(1)--(12)--(2)--(e)--(3)--(23)--(2);
		\draw(1)--(13)--(3)--(23)--(123)--(13)--(123)--(12);
	\end{scope}
	\end{tikzpicture}
	};
	\node at(7,0)
	{
	\begin{tikzpicture}
	\begin{scope}
		\fill(0,0)circle(1.5pt)coordinate(123);
		\fill(-1,-1)circle(1.5pt)coordinate(12);
		\fill(0,-1)circle(1.5pt)coordinate(13);
		\fill(1,-1)circle(1.5pt)coordinate(23);
		\fill(-1,-2)circle(1.5pt)coordinate(1);
		\fill(0,-2)circle(1.5pt)coordinate(2);
		\fill(1,-2)circle(1.5pt)coordinate(3);
		\fill(0,-3)circle(1.5pt)coordinate(e);
		\node[below]at(e){0};
		\node[left]at(1){1};
		\node[left]at(2){1};
		\node[right]at(3){2};
		\node[left]at(12){2};
		\node[right]at(13){2};
		\node[right]at(23){2};
		\node[right]at(123){2};
		\draw(e)--(1)--(12)--(2)--(e)--(3)--(23)--(2);
		\draw(1)--(13)--(3)--(23)--(123)--(13)--(123)--(12);
	\end{scope}
	\end{tikzpicture}
	};
	\node at (3,0)
	{
	\begin{tikzpicture}
	\begin{scope}
		\fill (0,0) circle (1.5pt) coordinate (123);
		\fill (-.75,-1.5) circle (1.5pt) coordinate (1);
		\fill (-.75,-.5) circle (1.5pt) coordinate (2);
		\fill (.75,-1) circle (1.5pt) coordinate (3);
		\fill (0,-2) circle (1.5pt) coordinate (0);
		\node[below]at(0){$\emptyset$};
		\node[left]at(1){$1$};
		\node[left]at(2){$12$};
		\node[right]at(3){$3$};
		\node[above]at(123){$123$};

		\draw (0)--(1)--(2)--(123)--(3)--(0);
		\draw (0)--(2);
		\draw (1)--(123);
	\end{scope}
	\end{tikzpicture}
	};
	\node at (10,0)
	{
	\begin{tikzpicture}
	\begin{scope}
		\fill (0,0) circle (1.5pt) coordinate (123);
		\fill (-.75,-1) circle (1.5pt) coordinate (1);
		\fill (.75,-1) circle (1.5pt) coordinate (2);
		\fill (0,-2) circle (1.5pt) coordinate (0);
		\node[below]at(0){$\emptyset$};
		\node[left]at(1){$1$};
		\node[right]at(2){$2$};
		\node[above]at(123){$123$};

		\draw (0)--(1)--(123)--(2)--(0)--(123);
	\end{scope}
	\end{tikzpicture}
	};
	\node at(0,-3){(a)};
	\node at(3,-3){(b)};
	\node at(7,-3){(c)};
	\node at(10,-3){(d)};
\end{tikzpicture}

	\caption{In (a) and (c) are polymatroids, and in (b) and (d)
	respectively are the diagrams of the generator enriched lattice of flats.}
	\label{fig genlatt of flats}
\end{figure}

Let~$r$ and~$s$ be two simple matroids with lattice of flats~$L$ and~$K$ respectively.
A \emph{strong map} between~$r$ and~$s$ is a map~$f:L\rightarrow K$ which
is join-preserving and satisfies~$f(\irr(L))\subseteq\irr(K)\cup\{\widehat{0}_K\}$.
Strong maps between simple matroids were introduced by Higgs in~\cite{higgs}.
The notion of structure preserving maps between generator enriched lattices defined below
generalizes strong maps between simple matroids.

\begin{definition}
	Let~$(L,G)$ and~$(K,H)$ be generator enriched lattices. A \emph{strong map}
	from~$(L,G)$ to~$(K,H)$ is a map~$f:L\rightarrow K$ which is
	join-preserving and satisfies~$f(G)\subseteq H\cup\{\widehat{0}_K\}$.
	This will be abbreviated by saying that~$f:(L,G)\rightarrow(K,H)$
	is a strong map.
\end{definition}

A strong map~$f:(L,G)\rightarrow(K,H)$ is said to be \emph{injective} when it
is injective as a map on the underlying lattices, and \emph{surjective}
when~$f(G\cup\{\widehat{0}_L\})=H\cup\{\widehat{0}_K\}$. Two generator enriched lattices are said
to be isomorphic when there is a strong bijection between them.

Strong maps between matroids
may be equivalently defined in several ways, for instance as join preserving
maps which also preserve the relation ``covers or equals'';
see~\cite[Proposition 2]{strongMapsCrapo}.
This definition does not extend to the setting of generator enriched lattices,
for example mapping atoms of a Boolean algebra to any elements of a chain
will induce a strong map which need not preserve covers.

Let~$(L,G)$ be a generator enriched lattice and let~$E$ be a ground set. Let~$\mathscr{B}_E$ denote the
generator enriched lattice~$(B_E,\irr(B_E))$.
Given any map~$f:E\rightarrow G\cup\{\widehat{0}_L\}$
we have an associated strong map~$F:\mathscr{B}_E\rightarrow(L,G)$
defined by
\[
F(X)=\bigvee_{x\in X}f(x),
\]
for~$X\subseteq E$. We refer to the map~$F$ as the \emph{strong map induced by~$f$}.

A certain nonstandard definition of matroids is useful for our lattice theoretic
view of polymatroids. A matroid on a ground set~$E$ may be defined as
a strong surjection~$f$ from the Boolean algebra~$\mathscr{B}_E$ onto a generator enriched lattice
of the form~$(L,\irr(L))$ for some geometric
lattice~$L$. In fact if one requires the map~$f$ to be strong in the sense
of~\cite{strongMapsCrapo}, the image is necessarily geometric;
see~\cite[Proposition 9.12]{combinatorialGeometries}.
This view of matroids is briefly mentioned in~\cite[pp. 9.8-9.9]{combinatorialGeometries}.
Accordingly, we now turn our focus to strong surjections from Boolean
algebras onto generator enriched lattices,
and showing such maps are in bijection with polymatroid closure operators
(when the codomain generator enriched lattice is considered up to isomorphism).

\subsection{Strong surjections and closure operators}

The following construction associates a closure operator to any
strong surjection from~$\mathscr{B}_E$ onto a generator enriched lattice~$(L,G)$.
This is a standard construction in the theory of Galois connections. Let~$\theta:\mathscr{B}_E\rightarrow(L,G)$
be a strong surjection.
Define a right-sided inverse~$\phi$ to~$\theta$ by
\[\phi(\ell)=\bigcup_{X\in\theta^{-1}(\ell)}X.\]
The fact that~$\theta\circ\phi$ is the identity follows directly from the fact that~$\theta$
is join-preserving. We define the closure operator associated to~$\theta$
to be the map~$\cl_\theta=\phi\circ\theta:B_E\rightarrow B_E$.
One may associate a generator enriched lattice~$(K,H)$ to such a closure operator by
setting
	\[K=\{\cl_\theta(X):X\subseteq E\},\]
and
	\[H=\{\cl_\theta(\{e\}):e\in E\}\setminus\{\cl_\theta(\emptyset)\}.\]
\begin{sloppypar}
This generator enriched lattice~$(K,H)$ is isomorphic to~$(L,G)$ via the isomorphism~$\phi:(L,G)\rightarrow
(K,H)$, which has inverse~$\phi^{-1}=\theta|_K$.
\end{sloppypar}

The following result says that a polymatroid can be equivalently defined as a strong surjection
from a Boolean algebra together
with a strictly order-preserving and submodular function with nonnegative real values.

\begin{proposition}
\label{polymatroids from genlatts}
	Let~$(L,G)$ be a generator enriched lattice and let~$\theta:\mathscr{B}_E\rightarrow(L,G)$
	be a strong surjection. For any strictly order-preserving and
	submodular function~$r:L\rightarrow\mathbb{R}_{\ge0}$ which maps~$\widehat{0}_L$
	to~$0$, the
	composition~$r\circ\theta:B_E\rightarrow\mathbb{R}_{\ge0}$
	is a polymatroid whose generator enriched lattice of flats is isomorphic
	to~$(L,G)$. Furthermore the polymatroid is simple if and only
	if~$\theta|_{\irr(B_E)\cup\{\emptyset\}}$ is injective.

	Conversely given a polymatroid~$s:B_E\rightarrow\mathbb{R}_{\ge0}$
	with generator enriched lattice of flats~$(L,G)$, let~$\theta:\mathscr{B}_E\rightarrow(L,G)$
	be the strong map induced by the map~$e\mapsto\overline{\{e\}}$.
	There is a strictly order-preserving and submodular function~$r:L\rightarrow\mathbb{R}_{\ge0}$
	such that~$s=r\circ\theta$.
\end{proposition}

	\begin{proof}
		Let~$s$ be the composition~$r\circ\theta$. 
		By assumption~$s(\emptyset)=r(\widehat{0})=0$.
		The maps~$\theta$ and~$r$ are order-preserving,
		hence~$s$ must be as well. To show that~$s$ is submodular,
		let~$X$ and~$Y$ be subsets of~$E$. Since~$\theta$ is join-preserving
		we have~$s(X\cup Y)=r(\theta(X)\vee\theta(Y))$. On the other
		hand since~$\theta$ is order-preserving, the image~$\theta(X\cap Y)$
		is a lower bound for both~$\theta(X)$ and~$\theta(Y)$,
		hence~$\theta(X\cap Y)\le\theta(X)\wedge\theta(Y)$.
		Thus~$s(X\cap Y)\le r(\theta(X)\wedge\theta(Y))$.
		Summing these two values results in the inequality
		\[
		s(X\cap Y)+s(X\cup Y)\le r(\theta(X)\wedge\theta(Y))+r(\theta(X)
		\vee\theta(Y)).
		\]
		Applying the submodularity of the function~$r$ leads to the inequality
		\[
		s(X\cap Y)+s(X\cup Y)\le r(\theta(X))+r(\theta(Y))
		=s(X)+s(Y).
		\]
		Therefore the function~$s$ is a polymatroid.

		To show that the generator enriched lattice of flats of~$s$ is isomorphic to~$(L,G)$,
		it will suffice to show that the closure operator~$\cl_\theta$ is the
		closure operator of~$s$. The closure of two sets~$X$ and~$Y$
		with respect to~$s$ is the same if and only if~$s(X)=s(X\cup Y)=s(Y)$.
		Since~$r$ is strictly order-preserving, this holds if and only
		if~$\theta(X)=\theta(Y)$, which holds if and only if~$\cl_\theta(X)
		=\cl_\theta(Y)$.
		By the same argument we see that~$s$ has a loop or a nontrivial
		parallel class precisely when~$\theta|_{\irr(B_E)\cup\{\emptyset\}}$
		is not injective.

		To prove the converse, consider a polymatroid~$s:B_E\rightarrow\mathbb{R}_{\ge0}$
		with generator enriched lattice of flats~$(L,G)$.
		Let~$\theta:\mathscr{B}_E\rightarrow(L,G)$ be the strong map
		induced by the map~$e\mapsto\overline{\{e\}}$,
		and let~$r=s|_L$.
		If~$A\subsetneq B\in L$ are flats then~$s(A)<s(B)$
		so~$r$ is strictly order-preserving on~$L$.
		Since~$A\vee B=\overline{A\cup B}$ we have~$s(A\vee B)=s(A\cup B)$.
		Therefore we have that~$r(A\wedge B)+r(A\vee B)=s(A\cap B)+s(A\cup B)$,
		which by submodularity of~$s$ is less than or equal to~$s(A)+s(B)$.
		This of course equals~$r(A)+r(B)$ so the function~$r$ is submodular.
	\end{proof}

\begin{lemma}
\label{submodular funcs exist}
	For any lattice~$L$ there exists a strictly order preserving submodular
	function~$r:L\rightarrow\mathbb{Z}_{\ge0}$ with~$r(\widehat{0})=0$.
\end{lemma}

	\begin{proof}
		It will suffice to construct such a function with values in~$\mathbb{Q}_{\ge0}$.
		Afterwards one can scale by a sufficiently large
		positive integer to clear denominators.
		Define a function~$r:L\rightarrow\mathbb{Q}_{\ge0}$
		by, for~$\ell\in L$ such that the largest
		chain in~$L$ from~$\widehat{0}$ to~$\ell$ is length~$k$,
		setting~$r(\ell)=1-2^{-k}$.
		The map~$r$ is strictly order-preserving
		and maps~$\widehat{0}$ to~$0$. To show~$r$ satisfies the submodularity
		condition, let~$x,y\in L$. It may be assumed that~$x\wedge y$
		is neither~$x$ nor~$y$, otherwise the submodularity inequality
		holds trivially for~$x$ and~$y$.
		Let~$r(x)=1-2^{-n}$ and~$r(y)=1-2^{-m}$.
		It may also be assumed that~$n\le m$.
		Observe that~$r(x\wedge y)\le 1-2^{-n+1}$ and~$r(x\vee y)\le 1$.
		Adding these terms gives, 
		\[
			r(x\wedge y)+r(x\vee y)\le 2-2^{-n+1}
			\le2-2^{-n}-2^{-m}
			=r(x)+r(y).
		\]
		Thus~$r$ is submodular and can be used to construct the desired function.
	\end{proof}

It is known that every lattice is isomorphic to the lattice of flats of some
polymatroid that is integer valued.
This result is attributed to Dilworth in \cite[pp. 26]{critProblems}
and follows from Dilworth's embedding theorem \cite[Theorem 14.1]{crawley-dilworth},
which states that any finite lattice can be embedded into a geometric lattice.
Below is a somewhat stronger result.

\begin{proposition}
	Every generator enriched lattice is isomorphic to the generator enriched lattice of flats of some polymatroid,
	which may be chosen to have integer values.
\end{proposition}

	\begin{proof}
		Let~$(L,G)$ be a generator enriched lattice.
		By Lemma~\ref{submodular funcs exist} there is an integer valued
		strictly order-preserving submodular function~$r$ on~$L$.
		Let~$\theta:\mathscr{B}_G\rightarrow(L,G)$ be the strong surjection
		induced by the identity map on~$G$.
		By Proposition~\ref{polymatroids from genlatts} the map~$r\circ\theta$
		is a polymatroid whose lattice of flats is isomorphic to~$(L,G)$.
	\end{proof}
	
\begin{theorem}
	Let~$E$ be a set. A function from~$B_E$ to~$B_E$ is the closure
	operator of a polymatroid if and only if it is the closure operator
	of a strong surjection~$\theta:\mathscr{B}_E\rightarrow(L,G)$
	onto some generator enriched lattice~$(L,G)$.
\end{theorem}

	\begin{proof}
		Let~$r:B_E\rightarrow\mathbb{R}_{\ge0}$ be a polymatroid,
		and let~$(L,G)$ be the generator enriched lattice of flats of~$r$.
		Let~$\theta:\mathscr{B}_E\rightarrow(L,G)$
		be the strong map induced by the map~$e\mapsto\overline{\{e\}}$
		from~$E$ to~$L$.
		The image~$\theta(X)$ is by definition
		\[\theta(X)=\bigvee_{x\in X}\overline{\{x\}},\]
		in other words the smallest flat including~$\overline{\{x\}}$
		for all~$x\in X$. If~$Y$ is a flat including~$\overline{\{x\}}$
		for all~$x\in X$, then~$Y\supseteq X$. Taking the closure~$Y
		\supseteq\overline{X}$. Thus the image~$\theta(X)$ equals the
		closure~$\overline{X}$. Since~$L\subseteq B_E$ the closure
		operator~$\cl_\theta$ takes the same values as~$\theta$
		so we have shown that~$\overline{\ \cdot\ }=\cl_\theta$.

		\begin{sloppypar}
		Conversely consider a generator enriched lattice~$(L,G)$ and a strong
		surjection~$\theta:\mathscr{B}_E\rightarrow(L,G)$.
		We wish to construct a polymatroid whose closure operator
		coincides with the closure operator~$\cl_\theta$ of~$\theta$.
		By~Lemma~\ref{submodular funcs exist} there is
		a strictly order-preserving submodular function~$r:
		L\rightarrow\mathbb{R}_{\ge0}$. By~Proposition~\ref{polymatroids from genlatts}
		the function~$s=r\circ\theta:B_E\rightarrow\mathbb{R}_{\ge0}$
		is a polymatroid on~$E$. Furthermore the generator enriched lattice
		of flats of~$s$ is isomorphic to~$(L,G)$ via the isomorphism~$\overline{X}
		\mapsto\theta(X)$. From this it is evident that~$\overline{\ \cdot\ }
		=\cl_\theta$, hence the closure operator~$\cl_\theta$ of~$\theta$
		is the closure operator of the polymatroid~$s$.\qedhere
		\end{sloppypar}
	\end{proof}

\section{Minors}
\label{lattice minors section}
	In this section we discuss minors of polymatroids in regards to the associated
closure operators and generator enriched lattices. The underlying generator enriched lattice
of a minor does not fully depend on the original polymatroid if it is simple,
only the underlying
generator enriched lattice. In Section~\ref{genlatt minors subsection} we discuss minors of generator enriched lattices themselves
(with no structure of a strong surjection). In Theorem~\ref{graph minors} we show
that for a graphic matroid the minors of the generator enriched lattice of flats are in bijection
with the minors of the graph when the vertices are labeled and the edges are unlabeled.
In Theorem~\ref{parallel closed pairs and lattice minors} we prove a generalization of this
result to polymatroids.

Let~$r$ be a polymatroid with ground set~$E$. The \emph{deletion by~$X\subseteq E$}
is the polymatroid~$r\setminus X:B_{E\setminus X}\rightarrow\mathbb{R}_{\ge0}$
defined as the usual function restriction~$r\setminus X=r|_{B_{E\setminus X}}$.
The \emph{contraction by~$X$} is the polymatroid~$r/X:B_{E\setminus X}\rightarrow\mathbb{R}_{\ge0}$
defined for~$Y\subseteq E\setminus X$ by setting~$(r/X)(Y)=r(Y\cup X)-r(X)$.
These operations correspond to restricting to a lower and upper interval of
the Boolean algebra~$B_E$ respectively. Any polymatroid obtained from~$r$
via deletion and contraction operations is said to be a \emph{minor} of~$r$.

\subsection{Minors of strong surjections}

We begin by observing that minors of a polymatroid closure operator are well defined
as the closure operator of the corresponding minor of any associated polymatroid.

\begin{lemma}
	Let~$(r,E)$ and~$(s,E)$ be two polymatroids with the same closure operator.
	For any two disjoint sets~$X,Y\subseteq E$ the closure operators
	of the minors~$(r/X)\setminus Y$ and~$(s/X)\setminus Y$ are the same.
\end{lemma}

	\begin{proof}
		Let~$r'=(r/X)\setminus Y$ and~$s'=(s/X)\setminus Y$.
		Let~$Z\subseteq E$ and~$e\in E\setminus Z$. By assumption~$r(Z)=r(Z\cup\{e\})$
		if and only if~$s(Z)=s(Z\cup\{e\})$. The minor~$r'$
		is the function defined on~$E\setminus(X\cup Y)$
		by~$r'(Z)=r(Z)-r(X)$, and similarly for~$s$ and~$s'$.
		Thus we have that~$r'(Z)=r'(Z\cup\{e\})$ if and only
		if~$s'(Z)=s'(Z\cup\{e\})$ which shows that the closure operators
		of~$r'$ and~$s'$ are the same.
	\end{proof}

We now turn to defining deletion and contraction operations on strong surjections
from a lattice theoretic viewpoint.
In Proposition~\ref{genlatt and polymatroid minors} we prove that these operations
agree with the same operations on polymatroids. First we set up some notation.

Given a lattice~$L$, let~$H\subseteq L$ and let~$z\in L$ be an element
such that~$z<h$ for all~$h\in H$. Define the \emph{generator enriched lattice with generating
set~$H$ and minimal element~$z$} to be
\begin{align*}\langle H|z\rangle&=\left(\{z\vee\bigvee_{x\in X}x:X\subseteq H\},H\right)\\
&=\left(\{z\}\cup\{\bigvee_{x\in X}x:\emptyset\ne X\subseteq H\},H\right).
\end{align*}
Usually when listing~$H$ explicitly the set brackets will be repressed.

Let~$E$ be a ground set, let~$(L,G)$ be a generator enriched lattice and let~$\theta:\mathscr{B}_E\rightarrow(L,G)$
be a strong surjection. For~$X\subseteq E$ define the \emph{deletion of~$\theta$ by~$X$}
to be the strong surjection
\[
\theta\setminus X:\mathscr{B}_{E\setminus X}\rightarrow
\langle\{\theta(\{e\}):e\in E\setminus X\}\setminus\{\widehat{0}_L\}|\widehat{0}_L\rangle,
\]
defined for~$Z\subseteq E\setminus X$
by~$(\theta\setminus X)(Z)=\theta(Z)$.
Define the \emph{contraction of~$\theta$ by~$X$} to be the strong surjection
\[
\theta/X:\mathscr{B}_{E\setminus X}\rightarrow
\langle\{\theta(X\cup\{e\}):e\in E\setminus X\}\setminus\{\theta(X)\}|\theta(X)\rangle,
\]
defined
for~$Z\subseteq E\setminus X$ by~$(\theta/X)(Z)=\theta(X\cup Z)$.

Conceptually the deletion by~$X$ of a strong surjection~$\theta$
is obtained by restricting~$\theta$ to the lower interval~$[\emptyset,E\setminus X]\subseteq B_E$,
and then restricting the codomain to ensure the resulting function is a surjection.
Similarly contracting by~$X$ corresponds to restricting to the upper
interval~$[X,E]\subseteq B_E$.

\begin{proposition}
\label{genlatt and polymatroid minors}
	Let~$r:E\rightarrow\mathbb{R}_{\ge0}$ be a polymatroid with generator enriched lattice of
	flats~$(L,G)$, and let~$X,Y\subseteq E$
	be disjoint sets. If~$\theta:\mathscr{B}_E\rightarrow(L,G)$ is the strong surjection
	associated to~$r$ then the closure operator of the polymatroid~$(r/X)\setminus Y$
	is equal to~$\cl_{(\theta/X)\setminus Y}$.
\end{proposition}

	\begin{proof}
		Set~$r'=(r/X)\setminus Y$ and~$\theta'=(\theta/X)\setminus Y$.
		Let~$\overline{\ \cdot\ }$ denote the closure operator of~$r'$.
		By definition~$\overline{Z_1}=\overline{Z_2}$ if and only
		if~$r'(Z_1)=r'(Z_1\cup Z_2)=r'(Z_2)$ which occurs if and only
		if~$r(Z_1\cup X)=r(Z_1\cup Z_2\cup X)=r(Z_2\cup X)$.
		Since~$\cl_\theta$ is the closure operator of~$r$ this occurs if
		and only if~$\theta(Z_1\cup X)=\theta(Z_1\cup Z_2\cup X)=\theta(Z_2\cup X)$.
		This is in turn equivalent to the condition~$\theta'(Z_1)=\theta'(Z_1\cup Z_2)
		=\theta'(Z_2)$.
		Therefore~$\overline{Z_1}=\overline{Z_2}$ if and only
		if~$\theta'(Z_1)=\theta'(Z_2)$, and thus~$\cl_{\theta'}$
		is the closure operator of~$r'$.
	\end{proof}

\subsection{Minors of generator enriched lattices}
\label{genlatt minors subsection}

Let~$(L,G)$ be a generator enriched lattice and~$\theta:\mathscr{B}_E\rightarrow(L,G)$ be a strong surjection.
When~$\theta$ is simple, that is, when~$\theta|_{\irr(B_E)\cup\{\emptyset\}}$
is injective, the codomain of the deletion~$\theta\setminus X$ depends only
on the set~$\{\theta(\{x\}):x\in X\}$. Similarly the codomain of
the contraction~$\theta/X$ depends only on the image~$\theta(X)$.
Thus viewing generator enriched lattices as encoding closure operators of simple polymatroids
we have a notion of deletion and contraction operations, the result of which
is another generator enriched lattice.

Let~$(L,G)$ be a generator enriched lattice and let~$I\subseteq G$. The \emph{deletion of~$(L,G)$ by~$I$}
is the generator enriched lattice
\[
(L,G)\setminus I=\langle G\setminus I|\widehat{0}_L\rangle.
\]
Let~$i_0=\bigvee_{i\in I}i$ and set~$J=\{g\vee i_0:g\in G\}\setminus\{i_0\}$.
The \emph{contraction of~$(L,G)$ by~$I$} is the generator enriched lattice
\[
(L,G)/I=\langle J|i_0\rangle.
\]
For convenience we also define the \emph{restriction of~$(L,G)$ to~$I$} as
\[
(L,G)|_I=(L,G)\setminus(G\setminus I).
\]
The operations of deletion and contraction on generator enriched lattices correspond to
first choosing a simple strong surjection, performing the operations as previously
defined for strong surjections, and taking a simplification of the result.

It will be convenient at times to index deletions and contractions
by subsets of some ground set~$E$,
or by elements of~$L$ instead. To define the former choose
a labeling of~$G$ by~$E$ so that~$G=\{g_e:e\in E\}$.
Given~$X\subseteq E$ the
deletion and contraction by~$X$ are defined as
\begin{align*}
(L,G)\setminus X&=(L,G)\setminus\{g_x:x\in X\},\\
(L,G)/X&=(L,G)/\{g_x:x\in X\}.
\end{align*}
Given~$\ell\in L$ the deletion and contraction by~$\ell$ are defined as
\begin{align*}
(L,G)\setminus \ell&=(L,G)\setminus\{g\in G:g\le\ell\},\\
(L,G)/\ell&=(L,G)/\{g\in G:g\le\ell\}.
\end{align*}

The result of any sequence of deletions and contractions applied to~$(L,G)$
is called a \emph{minor} of~$(L,G)$.

A few basic remarks for minors of a generator enriched lattice~$(L,G)$ are in order.

\begin{remark}
\label{first remark}
	By definition the underlying lattice of a minor of~$(L,G)$
	is a join subsemilattice of~$L$.
	In general the underlying lattice of a minor of~$(L,G)$
	may not be a sublattice of~$L$. For example,
	consider the partition lattice~$\Pi_4$ with minimal generating set 
	\[\irr(\Pi_4)=\{12/3/4,\ 13/2/4,\ 14/2/3,\ 1/23/4,\ 1/24/3,\ 1/2/34\}.\]
	Deleting the atom~$13/2/4$ results in a minor which is not a sublattice
	of~$\Pi_4$;
	in said minor the meet of~$123/4$ and~$134/2$ is the minimal partition~$1/2/3/4$
	as opposed to~$13/2/4$ when computed in~$L$.
\end{remark}

\begin{remark}
	Any interval of~$L$ is the underlying lattice of a minor of~$(L,G)$.
	If~$a\le b$ in~$L$
	then the minor~$((L,G)/a)|_b$ has underlying lattice the interval~$[a,b]$ of~$L$.
	The example given in Remark~\ref{first remark}
	shows the converse is false, that in general not all minors
	of~$(L,G)$ have as underlying lattice an interval of~$L$.
\end{remark}

\begin{remark}
	The deletion and
	contraction operations of generator enriched lattices do not in general commute.
	See Figure~\ref{del and contr not commuting figure} for an example.
\end{remark}

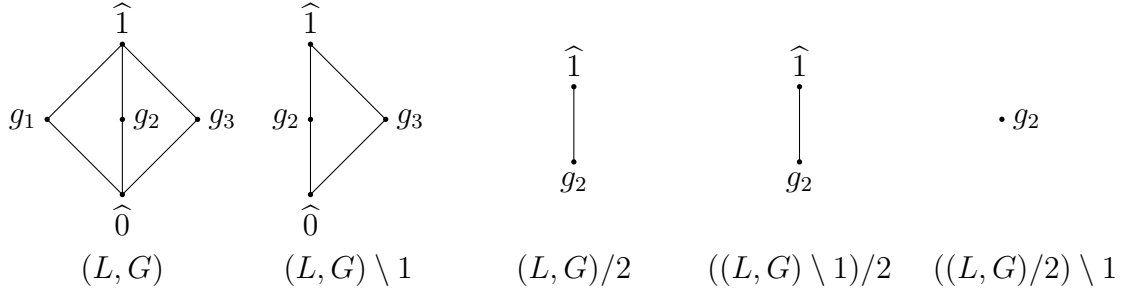
\begin{figure}
\centering
\begin{tikzpicture}
	\node at(0,0)
			{
			\begin{tikzpicture}\begin{scope}
				\fill(0,0)circle(1pt)coordinate(0);
				\fill(-1,1)circle(1pt)coordinate(g);
				\fill(0,1)circle(1pt)coordinate(h);
				\fill(1,1)circle(1pt)coordinate(i);
				\fill(0,2)circle(1pt)coordinate(1);
				\node[below]at(0){$\widehat{0}$};
				\node[left]at(g){$g_1$};
				\node[right]at(h){$g_2$};
				\node[right]at(i){$g_3$};
				\node[above]at(1){$\widehat{1}$};
				\draw(0)--(g)--(1)--(h)--(0)--(i)--(1);
			\end{scope}\end{tikzpicture}
			};
		\node at(0,-2){$(L,G)$};
		\node at(3,0)
			{
			\begin{tikzpicture}\begin{scope}
				\fill(0,0)circle(1pt)coordinate(0);
				\fill(0,1)circle(1pt)coordinate(h);
				\fill(1,1)circle(1pt)coordinate(i);
				\fill(0,2)circle(1pt)coordinate(1);
				\node[below]at(0){$\widehat{0}$};
				\node[left]at(h){$g_2$};
				\node[right]at(i){$g_3$};
				\node[above]at(1){$\widehat{1}$};
				\draw(0)--(h)--(1)--(i)--(0);
			\end{scope}\end{tikzpicture}
			};
		\node at(3,-2){$(L,G)\setminus 1$};
		\node at(6,0)
			{
			\begin{tikzpicture}\begin{scope}
				\fill(0,0)circle(1pt)coordinate(h);
				\fill(0,1)circle(1pt)coordinate(1);
				\node[below]at(h){$g_2$};
				\node[above]at(1){$\widehat{1}$};
				\draw(h)--(1);
			\end{scope}\end{tikzpicture}
			};
		\node at(6,-2){$(L,G)/2$};
		\node at(9,0)
			{
			\begin{tikzpicture}\begin{scope}
				\fill(0,0)circle(1pt)coordinate(h);
				\fill(0,1)circle(1pt)coordinate(1);
				\node[below]at(h){$g_2$};
				\node[above]at(1){$\widehat{1}$};
				\draw(h)--(1);
			\end{scope}\end{tikzpicture}
			};
		\node at(9,-2){$((L,G)\setminus 1)/2$};
		\node at(12,0)
			{
			\begin{tikzpicture}\begin{scope}
				\fill(0,0)circle(1pt)coordinate(h);
				\node[right]at(h){$g_2$};
			\end{scope}\end{tikzpicture}
			};
		\node at(12,-2){$((L,G)/2)\setminus 1$};

\end{tikzpicture}
\caption{A generator enriched lattice~$(L,G)$, where~$G=\{g_1,g_2,g_3\}$ for which deletions and contractions
do not commute, along with the relevant minors.}
\label{del and contr not commuting figure}
\end{figure}

The following observation will be useful.
\begin{lemma}
\label{contr then del}
Any minor of a generator enriched lattice~$(L,G)$ may be expressed as the result of a
contraction followed by a deletion. Namely, a minor~$(K,H)$ of~$(L,G)$ may be expressed
as~$(K,H)=((L,G)/\widehat{0}_K)|_{H}$.
\end{lemma}

	\begin{proof}
		Let~$(K,H)$ be a minor of~$(L,G)$.
		By definition~$(K,H)$ may be expressed as the result of a sequence
		of contractions and deletions. That is, for some possibly
		empty sets of generators~$I_1,J_1,
		\dots,I_r,J_r$, that
		\[(K,H)=((\cdots(((L,G)/I_1)\setminus J_1)\cdots/I_r)\setminus J_r.\]
		For~$1\le j\le r$ let~$i_j$ be the join of all elements
		in~$I_j$. Set~$i_0=i_1\vee\cdots\vee i_r$.
		By definition of deletion and contraction, the minimal element~$\widehat{0}_K$
		of~$K$ is~$i_0$. Furthermore, the generators of~$(K,H)$ can each
		be expressed as~$g\vee i_1\vee\cdots\vee i_r=g\vee i_0$
		for some~$g\in G$. Thus each generator of~$(K,H)$ is
		a generator of~$(L,G)/i_0$, hence~$(K,H)=((L,G)/i_0)|_{H}$.
	\end{proof}

The lemma below gives an explicit description of the generating sets of minors.

\begin{lemma}
\label{minor gen sets}
For any generator enriched lattice~$(L,G)$ the minors are precisely generator enriched lattices of the
form~$\langle\ell\vee g_1,\dots,\ell\vee g_k|\ell\rangle$
for~$\ell\in L$ and~$\{g_1,\ldots,g_k\}\subseteq G$ such that~$g_j\not\le\ell$
for~$1\le j\le k$.
\end{lemma}

	\begin{proof}
		Consider a minor~$(K,H)=((L,G)/I)|_J$ of~$(L,G)$,
		where~$I$ and~$J$ are sets of generators.
		Let~$\ell$ be the join of all elements of~$I$
		and let~$J=\{j_1,\dots,j_k\}$.
		By definition
		\[(K,H)=\langle\ell\vee j_1,\dots,\ell\vee j_k|\ell\rangle.\]
		Conversely consider a
		generator enriched lattice~$(K,H)=\langle\ell\vee g_1,\dots,\ell\vee g_k|\ell\rangle$
		for some~$\ell\in L$ and~$g_j\in G$ with~$g_j\not\le\ell$ for~$1\le j\le k$.
		The generators of the contraction~$(L,G)/\ell$ are all
		elements~$\ell\vee g$ for~$g\in G$ with~$g\not\le\ell$.
		Thus~$\ell\vee g_1,\dots,\ell\vee g_k$ are generators of~$(L,G)/\ell$, so
		setting~$I=\{\ell\vee g_1,\dots,\ell\vee g_k\}$
		we have that~$(K,H)=((L,G)/\ell)|_I$.
	\end{proof}

\begin{lemma}
\label{geometric minors}
If~$L$ is a geometric lattice then the minors of~$(L,\irr(L))$
are the generator enriched lattices of the form~$\langle \ell_1,\dots,\ell_k|\ell\rangle$
such that~$\ell_i\succ\ell\in L$
for~$1\le i\le k$.
In particular, every minor of~$(L,\irr(L))$
is minimally generated and geometric.
\end{lemma}

	\begin{proof}
		Since~$L$ is geometric,
		for any~$x,y\in L$ we have~$x\prec y$ if and only
		if~$y=x\vee i$ for some~$i\in\irr(L)$. Thus Lemma~\ref{minor gen sets}
		specializes to the claimed form of the generating sets
		of minors of~$(L,\irr(L))$. 
		In particular, for any minor~$(K,H)$ the generating set~$H$
		is the set of atoms of~$K$.

		Let~$(K,\irr(K))=\langle\ell_1,\dots,\ell_k|\ell\rangle$
		be a minor of~$(L,\irr(L))$. In order to show that~$K$
		is semimodular we claim that if~$x\prec y$ in~$K$
		then~$x\prec y$ in~$L$ as well.
		Since~$x\prec y$ there exists~$i$ such
		that~$y=x\vee\ell_i$. We have that~$\ell_i=\ell\vee a$
		for some atom~$a$ of~$L$, hence~$y=x\vee a$.
		Since~$L$ is geometric this implies that~$x\prec y$
		in~$L$.

		Now let~$x,y\in K$ such that~$x\wedge_Ky\prec x$
		and~$x\wedge_Ky\prec y$ in~$K$. Since~$K$ is a subposet
		of~$L$ we have that~$x\wedge_Ky\le x\wedge_Ly$. On the other
		hand, since~$x\wedge_K y$ is covered by~$x$ and~$y$ in~$K$,
		hence in~$L$, we must have~$x\wedge_Ky=x\wedge_Ly$.
		Thus~$x\wedge_Ly$ is covered by~$x$ and~$y$ in~$L$, and
		since~$L$ is semimodular~$x\vee_Ly$ covers~$x$ and~$y$ in~$L$.
		Then since~$x\vee_Ly=x\vee_Ky$ this implies
		that~$x\vee_Ky$ covers~$x$ and~$y$ in~$K$,
		and therefore~$K$ is geometric.
	\end{proof}

Given a graph~$G$, the lattice of flats~$L$ may be viewed as a
lattice of partitions of the vertices
of~$G$. Each flat is associated to the partition whose blocks consist of the connected
components of said flat considered as a subgraph of~$G$. Let~$L(G)$ denote the
generator enriched lattice of flats of~$G$ labeled as partitions.
See Figure~\ref{fig lattice of flats of graph}.

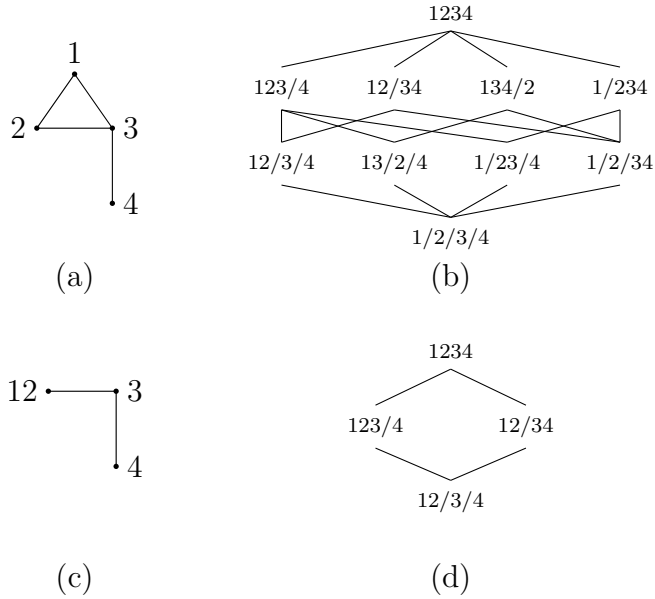
\begin{figure}
\centering
\begin{tikzpicture}
		\node at (0,0)
			{
			\begin{tikzpicture}\begin{scope}
				\node (1) at (0,0) {1};
				\node (2) at (-3/4,-1) {2};
				\node (3) at (3/4,-1) {3};
				\node (4) at (3/4,-2) {4};

				\fill (1.south) circle (1pt) -- (2.east) circle (1pt)
				-- (3.west) circle (1pt) -- (4.west) circle (1pt);

				\draw (1.south) -- (2.east) -- (3.west) -- (1.south);
				\draw (3.west)--(4.west);
			\end{scope}\end{tikzpicture}
			};
		\node at(0,-2){(a)};

		\node at (5,0)
			{
			\begin{tikzpicture}\begin{scope}
				\node (1234) at (0,0) {$\scriptstyle 1234$};
				\node (123/4) at (-9/4,-1) {$\scriptstyle 123/4$};
				\node (12/34) at (-3/4,-1) {$\scriptstyle 12/34$};
				\node (134/2) at (3/4,-1) {$\scriptstyle 134/2$};
				\node (1/234) at (9/4,-1) {$\scriptstyle 1/234$};
				\node (12/3/4) at (-9/4,-2) {$\scriptstyle 12/3/4$};
				\node (13/2/4) at (-3/4,-2) {$\scriptstyle 13/2/4$};
				\node (1/23/4) at (3/4,-2) {$\scriptstyle 1/23/4$};
				\node (1/2/34) at (9/4,-2) {$\scriptstyle 1/2/34$};
				\node (1/2/3/4) at (0,-3) {$\scriptstyle 1/2/3/4$};

				\draw (1/2/3/4.north) -- (12/3/4.south);
				\draw (1/2/3/4.north) -- (13/2/4.south);
				\draw (1/2/3/4.north) -- (1/23/4.south);
				\draw (1/2/3/4.north) -- (1/2/34.south);

				\draw (12/3/4.north) -- (123/4.south);
				\draw (12/3/4.north) -- (12/34.south);
				\draw (13/2/4.north) -- (123/4.south);
				\draw (13/2/4.north) -- (134/2.south);
				\draw (1/23/4.north) -- (123/4.south);
				\draw (1/23/4.north) -- (1/234.south);
				\draw (1/2/34.north) -- (12/34.south);
				\draw (1/2/34.north) -- (134/2.south);
				\draw (1/2/34.north) -- (1/234.south);

				\draw (123/4.north)--(1234.south);
				\draw (12/34.north)--(1234.south);
				\draw (134/2.north)--(1234.south);
				\draw (1/234.north)--(1234.south);
			\end{scope}\end{tikzpicture}
			};
		\node at(5,-2){(b)};
		\node at (0,-4)
			{
			\begin{tikzpicture}\begin{scope}
				\node (2) at (-3/4,-1) {12};
				\node (3) at (3/4,-1) {3};
				\node (4) at (3/4,-2) {4};

				\fill (2.east) circle (1pt)
				-- (3.west) circle (1pt) -- (4.west) circle (1pt);

				\draw (2.east) -- (3.west);
				\draw (3.west)--(4.west);
			\end{scope}\end{tikzpicture}
			};
		\node at(0,-6){(c)};
		\node at(5,-4)
			{
			\begin{tikzpicture}\begin{scope}
				\node(1234)at(0,0) {$\scriptstyle1234$};
				\node(123/4)at(-1,-1){$\scriptstyle123/4$};
				\node(12/34)at(1,-1){$\scriptstyle12/34$};
				\node(12/3/4)at(0,-2){$\scriptstyle12/3/4$};
				\draw(12/3/4.north)--(12/34.south);
				\draw(12/3/4.north)--(123/4.south);
				\draw(12/34.north)--(1234.south);
				\draw(123/4.north)--(1234.south);
			\end{scope}\end{tikzpicture}
			};
		\node at(5,-6){(d)};
	\end{tikzpicture}

\caption{At (a) a graph~$G$ and at (b) the lattice of flats~$L(G)$. At (c) the simplification
of the contraction~$G/\{1,2\}$ and at (d) the contraction~$L(G)/(12/3/4)$.}
\label{fig lattice of flats of graph}
\end{figure}
The minors of the graph~$G$ inherit a vertex labeling by blocks of a partition of the vertices
of~$G$. When an edge is contracted, the label of the new vertex is obtained by merging
the two blocks labeling the vertices of the contracted edge. In this way the minors
of a vertex labeled graph are considered to be themselves vertex labeled graphs.

\begin{theorem}
\label{graph minors}
	Let~$G$ be a vertex labeled graph with unlabeled edges.
	The vertex labeled minors of~$G$ which are simple are in bijection with the minors of
	the minimally generated lattice of flats~$L(G)$ via the map~$H\mapsto L(H)$.
\end{theorem}

	\begin{proof}
		Let~$L$ be the lattice of flats of the graph~$G$.
		It may be assumed that the graph~$G$ is simple, that is,
		that~$G$ has no loops or multiple edges. This only changes the labeling
		of elements in the lattice of flats and does not change
		the collection of simple minors of~$G$.
		The inverse of the map~$H\mapsto L(H)$ will be constructed.
		Let~$(K,\irr(K))=\langle\ell_1,\dots,\ell_k|\ell\rangle$ be a minor of~$(L,\irr(L))$.
		Construct a graph~$H$ as follows.
		Each atom of~$L$ corresponds to an edge of~$G$. The
		element~$\ell$ corresponds
		to a set of edges of~$G$; namely, those edges corresponding to an atom
		which is less than or equal to~$\ell$.
		Let~$H'$ be the minor of~$G$ obtained by contracting
		this set of edges corresponding to~$\ell$.
		The vertices of~$H'$ are labeled by the blocks of the partition~$\ell$.
		Each atom~$\ell_i$ in~$K$ is obtained from the partition~$\ell$
		by merging two blocks, and corresponds to an edge in~$H'$.
		Let~$H''$ be the graph obtained from~$H'$ by restricting to these
		edges which correspond to an atom of~$K$. The graph~$H$ is defined
		to be the simplification of~$H''$.

		It remains to show that the map~$K\mapsto H$ constructed above
		and the map~$H\mapsto L(H)$ are inverses. A vertex labeled
		graph~$H$ is determined by the labeling of its vertices and its edges.
		The associated lattice minor~$(K,\irr(K))$ of~$(L,\irr(L))$
		records this same information
		as the minimal partition and the atoms, which in turn
		determines~$K$.
	\end{proof}

The above result generalizes to polymatroids with the appropriate notion replacing
vertex labeled minors.

\begin{definition}
	Let~$r$ be a polymatroid with ground set~$E$.
	A \emph{parallel closed
	pair} is a pair~$(F,s)$ such that~$F\subseteq E$ is a flat of~$r$
	and~$s$ is a polymatroid that may be obtained
	as a deletion of the polymatroid~$r/F$ satisfying
	the following condition: 
	\begin{addmargin}{3em}
		If~$e\in E\setminus F$
		is parallel with respect to~$r/F$ to an element~$f$ of the ground set
		of~$s$ then~$e$
		is an element of the ground set of~$s$ as well. In other words,
		the ground set
		of~$s$ must be a union of parallel classes with respect to~$r/F$.
	\end{addmargin}
\end{definition}

For a graphic matroid the
parallel closed pairs are in bijection with the vertex labeled minors of the graph
obtained by first contracting,
and then deleting entire parallel classes of edges. The vertex labeling naturally
encodes the flat in the parallel closed pair.
Such graphs are in bijection
with the simple minors when the edges are unlabeled. Without an edge labeling
each such graph has one simplification, obtained by identifying parallel edges,
and no two such graphs have
the same simplification. Thus the following theorem is an analogue of
Theorem~\ref{graph minors}.

\begin{theorem}
\label{parallel closed pairs and lattice minors}
	Let~$r$ be a polymatroid and let~$(L,G)$ be the associated
	generator enriched lattice of flats.
	The minors of~$(L,G)$ are in bijection with the parallel
	closed pairs of~$r$.
\end{theorem}

	\begin{proof}
		Let~$E$ be the ground set of~$r$.
		Let~$\overline{\ \cdot\ }:B_E\rightarrow B_E$
		be the closure operator of~$r$. 
		Let~$\theta:\mathscr{B}_E\rightarrow(L,G)$ be the strong surjection
		induced by the ground set map~$e\mapsto\overline{\{e\}}$.

		Let~$(F,s)$ be a parallel closed pair of~$r$
		and let~$Y\subseteq E$ be the ground set of~$s$.
		Define a map~$f$ from the set of parallel closed pairs of
		the polymatroid~$r$ to the set of minors of the generator enriched lattice~$(L,G)$
		by
		\[f(F,s)=((L,G)|_{Y\cup F})/F.\]

		To show~$f$ is a bijection construct the inverse map~$g$.
		Let~$(K,H)$ be a minor of~$(L,G)$ and let~$Y$ be the
		set \[Y=\{y\in E:\overline{\{y\}\cup \widehat{0}_K}\in H\}.\]
		Let~$g(K,H)$ be the pair~$(\widehat{0}_K,(r/\widehat{0}_K)|_Y)$.
		Observe that~$g(K,H)$ is a parallel closed pair of~$r$.
		Furthermore~$g$ is the inverse of~$f$ so the map~$f$ is
		a bijection.
	\end{proof}

\subsection{Minors of distributive lattices}

In the remainder of this section we examine minors of minimally generated distributive
lattices. Recall that the fundamental theorem of finite distributive lattices
states that every finite distributive lattice~$L$ is isomorphic to the lattice of
lower order ideals of the subposet~$\irr(L)$ of~$L$.
The minors of a minimally generated distributive lattice
have an alternative description in terms of certain pairs of subsets
of the poset of irreducibles.

\begin{definition}
	Let~$P$ be a poset. An \emph{order minor} of~$P$ is a pair~$(I,J)$
	of disjoint subsets of~$P$ such that~$J$ is a lower order ideal of~$P$.
\end{definition}

The poset~$P$ itself corresponds to the order minor~$(P,\emptyset)$.
The set of order minors of~$P$ is shown below to be in bijection with the minors
of the minimally generated lattice of lower order ideals of~$P$.
To prove this bijection, the following lemma is needed.

\begin{lemma}
\label{distr no paras}
If~$L$ is a distributive lattice, for any~$\ell\in L$ and any distinct join irreducibles~$i$
and~$j$ such that~$i\vee\ell\ne\ell$ and~$j\vee\ell\ne\ell$ 
the elements~$i\vee\ell$ and~$j\vee\ell$ are distinct.
\end{lemma}

	\begin{proof}
		Let~$L$ be the lattice of lower order ideals of
		a poset~$P$, necessarily isomorphic to~$\irr(L)$.
		It may be assumed without loss of generality that~$i\not\le j$ in~$L$.
		Let~$p\in P$ be the element such that the principal lower
		order ideal of~$P$ generated by~$p$ is the join irreducible~$i$
		of~$L$. The fact that~$i\vee\ell\ne\ell$ implies that~$p$ is
		not contained in the ideal~$\ell$. Since~$i\not\le j$
		the element~$p$ is not contained in the ideal~$j$.
		As a consequence~$p$ is not contained in the ideal~$j\vee\ell$
		since the join in~$L$ corresponds to the union of lower order ideals.
		This establishes that~$i\vee\ell\ne j\vee\ell$.
	\end{proof}

When~$(L,\irr(L))$ is the minimally generated lattice of lower order ideals of
a poset~$P$ there is an implicit bijection between~$P$ and the generating
set~$\irr(L)$. Through this bijection deletions and contractions of~$(L,\irr(L))$ may be
indexed by subsets of~$P$.

\begin{proposition}
\label{order minors bijection}
Let~$L$ be the lattice of lower order ideals of a poset~$P$.
The order minors of~$P$ and the minors of~$(L,\irr(L))$ are in bijection via
the map
\[(I,J)\mapsto((L,\irr(L))|_{I\cup J})/J.\]
\end{proposition}

	\begin{proof}
		Define a map from minors of~$(L,\irr(L))$ to order minors of~$P$
		as follows. Given a minor~$(K,H)$ of~$(L,\irr(L))$ define~$J$ to be the
		subset of~$P$ corresponding to the join irreducibles in~$L$
		which are less than or equal to~$\widehat{0}_K$.
		Define~$I$ to be the subset of~$P$ consisting of all elements
		whose corresponding join irreducible~$i$ of~$L$
		satisfies~$i\vee\widehat{0}_K\in H$. Lemma~\ref{distr no paras} implies that~$I$
		is the unique set satisfying~$(K,H)=((L,\irr(L))|_{I\cup J})/J$.
		The inverse of this map is given by~$(I,J)\mapsto((L,\irr(L)|_{I\cup J})/J$
		so we have a bijection.
	\end{proof}

Not only do the order minors of a poset index the minors of the minimally
generated lattice of lower order ideals, 
the order minors also describe the isomorphism types of the lattice minors.

\begin{proposition}
\label{distr minors}
	Let~$P$ be a poset and~$L$ the lattice of lower order
	ideals of~$P$, and let~$(I,J)$ be an order minor of~$P$.
	The minor~$((L,\irr(L))|_{I\cup J})/J$ of~$(L,\irr(L))$
	is isomorphic to the minimally generated lattice of lower
	order ideals of~$I$.
\end{proposition}

	\begin{proof}
		Let~$(I,J)$ be an order minor of~$P$
		and let~$(K,H)=((L,\irr(L))|_{I\cup J})/J$.
		It is claimed that~$K$ consists of the lower order ideals of~$P$
		whose maximal elements are all contained in~$I\cup J$ and which include~$J$.
		Observe that~$(L,\irr(L))|_{I\cup J}$ is generated by the principal lower
		order ideals which are themselves generated by an element of~$I\cup J$.
		Hence this lattice consists of all lower order ideals of~$P$
		whose maximal elements are contained in~$I\cup J$.
		The generators of~$((L,\irr(L))|_{I\cup J})/J$ are thus each the union of
		the lower order ideal~$J$ of~$P$ with a principal lower
		order ideal of~$P$ which is generated by an element of~$I$.
		Such lower order ideals as a join subsemilattice of~$L$
		generate the set of lower order ideals of~$P$ which include~$J$
		and whose maximal elements are contained in~$I\cup J$.
		
		Let~$(M,\irr(M))$
		be the minimally generated lattice of lower order ideals
		of the subposet~$I$ of~$P$. Define a map~$f:(K,H)\rightarrow(M,\irr(M))$
		by~$f(\Lambda)=\Lambda\cap I$.
		Define a map~$g:(M,\irr(M))\rightarrow(K,H)$
		by letting~$g(\Lambda)$ be the lower order ideal of~$P$
		generated by~$\Lambda\cup J$. Observe this is the inverse of~$f$ since
		every ideal in~$L$ includes~$J$ and
		has maximal elements which are contained in~$I\cup J$.
		Therefore~$K$ is isomorphic to~$M$ as claimed.
	\end{proof}

\section{Acknowledgments}
	The author thanks Richard Ehrenborg and Margaret Readdy for extensive comments.

\bibliography{bib}{}
\bibliographystyle{plain}

\end{document}